\documentclass[12pt,reqno,a4paper]{amsart}
\usepackage{amsfonts}
\usepackage{amsmath}
\usepackage{breqn}

\newtheorem{theorem}{Theorem}[section]
\newtheorem{corollary}[theorem]{Corollary}
\newtheorem{lemma}[theorem]{Lemma}
\theoremstyle{definition}
\newtheorem{definition}[theorem]{Definition}

\theoremstyle{example}

\begin{document}
\title[Spirallike univalent functions ...]{Certain subclasses of Spirallike univalent functions
related with Pascal distribution series}
\author{G. Murugusundaramoorthy }
\address{School of Advanced Sciences, Vellore Institute of Technology,
deemed to be university Vellore - 632014, Tamilnadu, India}
\email{gmsmoorthy@yahoo.com}
\begin{abstract}
The purpose of the present paper is to find the necessary and sufficient
conditions for the subclasses of analytic functions associated with Pascal distribution to be in subclasses of  spiral-like univalent functions  and inclusion relations for such subclasses  in the open unit disk
$\mathbb{D}.$ Further, we consider the properties of  integral operator related to Pascal
distribution series. Several corollaries and consequences of the main
results are also considered.


\end{abstract}

\maketitle

\section{\protect\bigskip Introduction and definitions}
Denote by $\mathcal{A}$  the class of functions whose members are of the form %
\begin{equation}
f(z)=z+\sum_{n=2}^{\infty }a_{n}z^{n},  \label{e1.1}
\end{equation}%
which are analytic in the open unit disk $\mathbb{D}=\{z\in \mathbb{C}%
:\left\vert z\right\vert <1\}$ and normalized by the conditions $%
f(0)=0=f^{\prime }(0)-1.$ Let $\mathcal{S}$  be subclass of $\mathcal{A}$ whose members are given by \eqref{e1.1} and are univalent in $\mathbb{D}.$
For functions $f\in \mathcal{S}$  be given by (\ref{e1.1}) and $g\in \mathcal{S}$
given by $g(z)=z+\sum_{n=2}^{\infty }b_{n}z^{n},$ we define the Hadamard
product (or convolution) of $f$ and $g$ by
\begin{equation*}
(f\ast g)(z)=z+\sum\limits_{n=2}^{\infty }a_{n}b_{n}z^{n},\,\,\,\,z\in
\mathbb{D}.
\end{equation*}The two well known subclass of $\mathcal{S},$ are namely the class of starlike and convex functions (for details see Robertson \cite{10}).
A function $f\in \mathcal{S}$ is said to be starlike of order $\gamma $ $%
(0\leq \gamma <1),$ if and only if
\begin{equation*}
\mbox{Re}\left( \frac{zf^{\prime }(z)}{f(z)}\right) >\gamma \quad (z\in
\mathbb{D}).
\end{equation*}%
This function class is denoted by $\mathcal{S}^{\ast }(\gamma ).$ We also
write $\mathcal{S}^{\ast }(0)=:\mathcal{S}^{\ast },$ where $\mathcal{S}%
^{\ast }$ denotes the class of functions $f\in \mathcal{A}$ that $f(\mathbb{D%
})$ is starlike with respect to the origin.

A function $f\in \mathcal{S}$ is said to be convex of order $\gamma $ $%
(0\leq \gamma <1)$ if and only if
\begin{equation*}
\mbox{Re}\left( 1+\frac{zf^{\prime \prime }(z)}{f^{\prime }(z)}\right)
>\gamma \quad (z\in \mathbb{D}).
\end{equation*}%
This class is denoted by ${\mathcal{K}}(\gamma ).$ Further, ${\mathcal{K}}={%
\mathcal{K}}(0)$, the well-known standard class of convex functions.
By Alexander's relation, it is a known fact that  $$f \in \mathcal{K}\Leftrightarrow zf'(z) \in
\mathcal{S}^*.$$
A function $f\in \mathcal{S}$ is said to be spirallike  if
\begin{equation*}
\Re \left( e^{-i\alpha}\frac{zf^{\prime }(z)}{f(z)}\right) >0
\end{equation*}
for some $\alpha$ with $\mid\alpha\mid<\frac{\pi}{2}$ and for all $z\in\mathbb{D}$ this class of spiral-like function was introduced by\cite{spa}. Also $f(z)$ is convex spiral-like if $zf'(z)$ is spiral-like. Due to Murugusundramoorthy\cite{GMS1,GMS2}, we consider subclasses of spiral-like  functions as below:

\begin{definition}\label{defa}
For $0 \leq \rho <1,$  $0 \leq \gamma < 1$  then
\begin{eqnarray*}
&& \mathcal{S}(\xi,
\gamma,\rho):= \left \{ f \in \mathcal{S} : \mbox{Re~} \left (
e^{i \xi}\frac{zf'(z)}{(1-\rho)f(z) +\rho zf'(z)}\right )
> \gamma \cos \xi, \,\,\, |\xi|< \frac{\pi}{2},\,\,\,  z \in \mathbb{D}\right
\}.
\end{eqnarray*}
\end{definition}

By virtue of Alexander's relation, we define the following subclass:
\begin{definition}\label{defb}
For  $0 \leq\rho <1,$ $0 \leq \gamma < 1$  then
\begin{eqnarray*}
&& \mathcal{K}(\xi,
\gamma,\rho):= \left \{ f \in \mathcal{S} : \mbox{Re~} \left (
e^{i \xi}\frac{zf''(z)+f'(z)}{f'(z) +\rho zf''(z)}\right )
> \gamma \cos \xi, \,\,\, |\xi|< \frac{\pi}{2},\,\,\,  z \in \mathbb{D}\right
\}.
\end{eqnarray*}
\end{definition}
By specialising the parameter $\rho=0$ we remark the following :

\begin{definition}\label{defa1}
For $0 \leq \gamma < 1$  then
\begin{eqnarray*}
&& \mathcal{S}(\xi,
\gamma):= \left \{ f \in \mathcal{S} : \mbox{Re~} \left (
e^{i \xi}\frac{zf'(z)}{f(z) }\right )
> \gamma \cos \xi, \,\,\, |\xi|< \frac{\pi}{2},\,\,\,  z \in \mathbb{D}\right
\}\\
{\text{and}} \\
&& \mathcal{K}(\xi,
\gamma):= \left \{ f \in \mathcal{S} : \mbox{Re~} \left (
e^{i \xi}\left[1+\frac{zf''(z)}{f'(z)}\right]\right )
> \gamma \cos \xi, \,\,\, |\xi|< \frac{\pi}{2},\,\,\,  z \in \mathbb{D}\right
\}.
\end{eqnarray*}
\end{definition}
Now we state the necessary sufficient conditions for $f$ in the above classes.
\begin{lemma}\cite{GMS1,GMS2}\label{lem1}
A function $f(z)$ of the form (\ref{e1.1}) is in
$\mathcal{S}(\xi, \gamma,\rho)$ if
\begin{equation}\label{GMS1}
\sum\limits_{n=2}^{\infty}   [ (1-\rho)(n-1) \sec \xi +
(1-\gamma) (1+n\rho -\rho)  ] |a_n| \leq 1-\gamma,
\end{equation}
where $|\xi| < \frac{\pi}{2},$ $0 \leq\rho <1,$ $0 \leq
\gamma <1.$
\end{lemma}
\begin{lemma}\label{lem2}
A function $f(z)$ of the form (\ref{e1.1}) is in
$\mathcal{S}(\xi, \gamma,\rho)$ if
\begin{equation}\label{GMS2}
\sum\limits_{n=2}^{\infty}  n [ (1-\rho)(n-1) \sec \xi +
(1-\gamma) (1+n\rho -\rho)  ] |a_n| \leq 1-\gamma,
\end{equation}
where $|\xi| < \frac{\pi}{2},$ $0 \leq\rho <1,$ $0 \leq
\gamma <1.$
\end{lemma}
\begin{proof}
  Using Alexander type Theorem which states that, If $f \in \mathcal{K}(\xi, \gamma,\rho)$ if and only if $zf^{\prime}\in \mathcal{S}(\xi, \gamma,\rho),$ Thus $z+\sum\limits_{n=2}^{\infty}  n a_n z^n $ is in $\mathcal{K}(\xi, \gamma,\rho).$ Hence by wringing $a_{n}$  in Lemma \ref{lem1} by $n a_{n} $ we get the desired result.
\end{proof}

\begin{lemma}\label{lem1a}
A function $f(z)$ of the form (\ref{e1.1}) is in
$\mathcal{S}(\xi, \gamma)$ if
\begin{equation}\label{GMS1a}
\sum\limits_{n=2}^{\infty}   [ (n-1) \sec \xi +
(1-\gamma) ] |a_n| \leq 1-\gamma,
\end{equation}
where $|\xi| < \frac{\pi}{2},$ $0 \leq
\gamma <1.$
\end{lemma}
\begin{lemma}\label{lem2a}
A function $f(z)$ of the form (\ref{e1.1}) is in
$\mathcal{K}(\xi, \gamma)$ if
\begin{equation}\label{GMS2a}
\sum\limits_{n=2}^{\infty}  n [(n-1) \sec \xi +
(1-\gamma) ] |a_n| \leq 1-\gamma,
\end{equation}
where $|\xi| < \frac{\pi}{2},$  $0 \leq
\gamma <1.$
\end{lemma}
\begin{definition}
\par A function $f \in \mathcal{S}$ is said to be in the class $\mathcal{R}
^{\tau } (\vartheta, \delta)$, $(\tau \in \mathbb{C} \backslash \{0\},\,\, 0 < \vartheta\leq 1;\delta < 1
)$,  if it satisfies the inequality
\begin{equation*}
  \left| \frac{(1-\vartheta)\frac{f(z)}{z}+\vartheta f^\prime(z)  - 1}
  { 2\tau(1-\delta)+ (1-\vartheta)\frac{f(z)}{z}+\vartheta f^\prime(z)  - 1} \right| < 1 \ \ \ (z \in \mathbb{U}).
\end{equation*}
\par The class $\mathcal{R} ^{\tau } (\vartheta, \delta)$ was introduced earlier by
Swaminathan \cite{swaminathan}(for special cases see the references cited there in).
\end{definition}
\begin{lemma}\cite{swaminathan}\label{lem4} If
$f\in \mathcal{R} ^{\tau } (\vartheta,\delta)$ is of form \eqref{e1.1}, then
\begin{equation}\label{dixcondi}
   \left| a_{n} \right| \leq \frac{2\left| \tau \right|(1-\delta) }{1+\vartheta(n-1)}, \quad
   n \in \mathbb{N} \setminus\{1\}.
\end{equation}
The bounds given in (\ref{dixcondi}) is sharp.
\end{lemma}

\par A variable $x$ is said to be \emph{Pascal distribution} if it takes the
values $0,1,2,3,\dots $ with probabilities
\newline $(1-q)^{m}$, $\dfrac{qm(1-q)^{m}}{%
1!}$, $\dfrac{q^{2}m(m+1)(1-q)^{m}}{2!}$, $\dfrac{q^{3}m(m+1)(m+2)(1-q)^{m}}{%
3!}$, \dots respectively, where $q$ and $m$ are called the parameters, and
thus
\begin{equation*}
P(x=k)=\binom{k+m-1}{m-1}%
\,q^{k}(1-q)^{m},k=0,1,2,3,\dots .
\end{equation*}
Lately, El-Deeb \cite{pascal}(also see \cite{GMSTB}) introduced a power series whose coefficients are probabilities of Pascal
distribution
\begin{equation*}
\Theta_{q}^{m}(z)=z+\sum\limits_{n=2}^{\infty }\binom{n+m-2}{m-1}
q^{n-1}(1-q)^{m}z^{n},\qquad z\in \mathbb{D}  \label{PHI}
\end{equation*}%
where $m\geq 1;0\leq q\leq 1$ and one can easily verify that the radius of
convergence of above series is infinity by ratio test.
Now, we define the linear operator
\begin{equation*}
\Lambda_{q}^{m}(z):\mathcal{A}\rightarrow \mathcal{A}
\end{equation*}%
defined by the convolution or Hadamard product
\begin{equation*}
\Lambda_{q}^{m}f(z)=\Theta_{q}^{m}(z)\ast
f(z)=z+\sum\limits_{n=2}^{\infty }\binom{n+m-2}{m-1}
q^{n-1}(1-q)^{m}a_{n}z^{n},\qquad z\in \mathbb{D}.  \label{I}
\end{equation*}

Inspired by earlier results on relations between different
subclasses of analytic and univalent functions by using hypergeometric
functions (see for example,\cite{1,fr4,2,8,9,swaminathan}) and by the recent
investigations related with distribution series (see for example, \cite{GMSTB,ou,pascal, fra, mur1, mur2, por1, por3, por2}, we obtain necessary and sufficient condition for the function $\Phi_q^m$ to be in the classes  $\mathcal{S}(\xi, \gamma,\rho)$ and $\mathcal{K}(\xi, \gamma,\rho)$, and information regarding the images of functions belonging in $\mathcal{R} ^{\tau } (\vartheta,\delta)$ by applying convolution operator. Finally, we provide conditions for the integral operator $\mathcal{G}^{m}_{q}(z)=%
\int_{0}^{z}\frac{\Theta_{q}^{m}(t)}{t}dt$ belonging to the above
classes.

\section{The necessary and sufficient conditions}

In order to prove our main results, we will use the following notations,
for $m \geq 1$ and $0 \leq q <1$:
\begin{eqnarray}\label{GMS1}
\sum\limits_{n=0}^{\infty} \binom{n+m-1}{m-1} q^{n} &=& \frac{1}{(1-q)^m};\quad
\sum\limits_{n=0}^{\infty} \binom{n+m}{m} q^{n} = \frac{1}{(1-q)^{m+1}}\nonumber \\ \quad {\text{and}}\quad
\sum\limits_{n=0}^{\infty} \binom{n+m+1}{m+1} q^{n} &=& \frac{1}{(1-q)^{m+2}}.
\end{eqnarray}
By simple computation we get the following relations:
\begin{equation}\label{GMS3}
\sum\limits_{n=2}^{\infty}\binom{n+m-2}{m-1}\,q^{n-1}=\sum\limits_{n=0}^{\infty}\binom{n+m-1}{m-1}\,q^{n}-1 = \frac{1}{(1-q)^m} - 1
\end{equation}
\begin{equation}\label{GMS4}
\sum\limits_{n=2}^{\infty}(n-1) \binom{n+m-2}{m-1}\,q^{n-1}=q m\sum\limits_{n=0}^{\infty}\binom{n+m}{m}\,q^{n} = \frac{q m}{(1-q)^{m+1}}
\end{equation}and
\begin{align}  \label{GMS5}
\sum\limits_{n=2}^{\infty}(n-1)(n-2)\binom{n+m-2}{m-1}\,q^{n-1}&=
q^2~m(m+1)\sum\limits_{n=0}^{\infty}\binom{n+m+1}{m+1}\,q^{n}\nonumber\\
&= \frac{q^2~m(m+1)}{(1-q)^{m+2}}.
\end{align}
\begin{theorem}
\label{th1}If $m>0,$ then $\Theta^m_q(z)\in\mathcal{S}(\xi, \gamma,\rho)$ if
\begin{eqnarray}\label{d2}
&&[(1-\rho)\sec \xi+\rho(1-\gamma)]\frac{q m}{(1-q)^{m+1}}\leq1-\gamma.
 \end{eqnarray}
\end{theorem}
\begin{proof}
Since%
\begin{equation*}
\Theta^m_q(z) = z+\sum\limits_{n=2}^{\infty }\binom{n+m-2}{m-1}
q^{n-1}(1-q)^mz^{n}.  \label{n5}
\end{equation*}%
Using the Lemma \ref{lem1}, it suffices to show that%
\begin{equation}
\sum\limits_{n=2}^{\infty }[ (1-\rho)(n-1) \sec \xi +
(1-\gamma) (1+n\rho -\rho)  ]\leq 1-\gamma .  \label{n1}
\end{equation}%
From (\ref{n1}) we let
\begin{eqnarray*}
M_1(\xi, \gamma,\rho)&=&\sum\limits_{n=2}^{\infty }[ (1-\rho)\sec \xi(n-1) +
(1-\gamma) (1+n\rho -\rho)  ]\binom{n+m-2}{m-1}
q^{n-1}(1-q)^m \\
&=&[(1-\rho)\sec \xi+\rho(1-\gamma)](1-q)^m\sum\limits_{n=2}^{\infty }(n-1)\nonumber\\ &\qquad\times& \binom{n+m-2}{m-1%
} q^{n-1}+(1-\gamma)(1-q)^m\sum\limits_{n=2}^{\infty }\binom{n+m-2}{m-1%
} q^{n-1} \\
&=&[(1-\rho)\sec \xi+\rho(1-\gamma)](1-q)^mq m\sum\limits_{n=0}^{\infty}\binom{n+m}{m}\,q^{n}\\
&\qquad+&(1-\gamma)(1-q)^m\left(\sum\limits_{n=0}^{\infty}\binom{n+m-1}{m-1}\,q^{n}-1\right) \\
&=&[(1-\rho)\sec \xi+\rho(1-\gamma)](1-q)^m\frac{q m}{(1-q)^{m+1}}\\
&\qquad+&(1-\gamma)(1-q)^m\left(\frac{1}{(1-q)^m} - 1\right)\\
&=&[(1-\rho)\sec \xi+\rho(1-\gamma)]\frac{q m}{(1-q)}+(1-\gamma)\left(1-(1-q)^m \right).
\end{eqnarray*}
But $M_1(\xi, \gamma,\rho)$ is bounded above by $1-\gamma$ if and
only if~(\ref{d2}) holds.
\end{proof}
\begin{theorem} \label{th11}If $m>0,$ then $\Theta^m_q(z)\in \mathcal{K}(\xi, \gamma,\rho)$ if
\begin{eqnarray}\label{t2}
&~&[(1-\rho)\sec \xi+(1-\gamma)]\frac{m(m+1)q^{2}}{(1-q)^2}+\left[2(1-\rho)sec\xi+(1-\gamma)(4-\rho)\right] \frac{mq}{1-q}\nonumber\\
&+&\left[(1-\gamma)(2 -\rho)\right]\left(1-(1-q)^m \right)\leq 1-\gamma.
\end{eqnarray}
\end{theorem}
\begin{proof}
In view of Lemma \ref{lem2}, we have to show that%
\begin{equation}
\sum\limits_{n=2}^{\infty }n [ (1-\rho)(n-1) \sec \xi +
(1-\gamma) (1+n\rho -\rho)  ]\binom{n+m-2}{m-1}
q^{n-1}(1-q)^m\leq 1-\gamma .  \label{we}
\end{equation}%
Writing  $n=(n-1)+1$ and
$n^{2}=(n-1)(n-2)+3(n-1)+1.$
\\From (\ref{we}), consider the expression
\begin{eqnarray*}
M_2(\xi, \gamma,\rho)&=&\sum\limits_{n=2}^{\infty}  n [ (1-\rho)(n-1) \sec \xi +
(1-\gamma) (1+n\rho -\rho)  ]\\&\qquad\times&\binom{n+m-2}{m-1}%
 q^{n-1}(1-q)^m \\
&=&[(1-\rho)\sec \xi+(1-\gamma)](1-q)^m\sum\limits_{n=2}^{\infty }n^2\binom{n+m-2}{m-1%
} q^{n-1}\\&\qquad-&(1-\rho)[sec\xi-(1-\gamma)(1-q)^m\sum\limits_{n=2}^{\infty }n\binom{n+m-2}{m-1%
} q^{n-1}\\
&=&[(1-\rho)\sec \xi+(1-\gamma)](1-q)^m\sum\limits_{n=2}^{\infty }(n-1)(n-2)\binom{n+m-2}{m-1}q^{n-1}\\
&\qquad+&\left[2(1-\rho)sec\xi+(1-\gamma)(4-\rho)\right](1-q)^m\sum\limits_{n=2}^{\infty }(n-1)\binom{n+m-2}{m-1}q^{n-1}\\
&\qquad+&\left[(1-\gamma)(2 -\rho)\right](1-q)^m\sum\limits_{n=2}^{\infty }\binom{n+m-2}{m-1}q^{n-1}\\
&=&[(1-\rho)\sec \xi+(1-\gamma)](1-q)^mq^{2}m(m+1)\sum\limits_{n=0}^{\infty }\binom{n+m+1}{m+1}q^{n}\\
&\qquad+&\left[2(1-\rho)sec\xi+(1-\gamma)(4-\rho)\right](1-q)^m mq\sum\limits_{n=0}^{\infty }\binom{n+m}{m}q^{n}
\end{eqnarray*}
\begin{eqnarray*}
&+&\left[(1-\gamma)(2 -\rho)\right](1-q)^m\left[\sum\limits_{n=0}^{\infty }\binom{n+m-1}{m-1}q^{n}-1\right].
 \end{eqnarray*}
 Now by using \eqref{GMS3}-\eqref{GMS5}, we get
 \begin{eqnarray*}
 M_2(\xi, \gamma,\rho)&=&[(1-\rho)\sec \xi+(1-\gamma)]\frac{m(m+1)q^{2}}{(1-q)^2}\\
&\qquad+&\left[2(1-\rho)sec\xi+(1-\gamma)(4-\rho)\right] \frac{mq}{1-q}\\
&\qquad+&\left[(1-\gamma)(2 -\rho)\right]\left(1-(1-q)^m \right).
 \end{eqnarray*}
Hence, $M_2(\xi, \gamma,\rho)$ is bounded above by $1-\gamma$if (\ref{t2}) is satisfied.
\end{proof}

\section{Inclusion Properties}
Making use of the Lemma \ref{lem4}, we will focus the influence of the Pascal distribution series on the classes $\mathcal{S}(\xi, \gamma,\rho)$ and $\mathcal{K}(\xi, \gamma,\rho)$.
\begin{theorem}\label{th2}If $\ f\in \mathcal{R} ^{\tau } (\vartheta,\delta)$ then $\Lambda_q^mf(z)$
is in $\mathcal{S}(\xi, \gamma,\rho)$ if and only if
\begin{eqnarray}\label{d3}
&&\frac{2\left| \tau \right|(1-\delta) }{\vartheta} \left\{\frac{}{}\left[(1-\rho)sec\xi+\rho(1-\gamma)\right]\left[1-(1-q)^{m}\right]\nonumber\right.\\
&+&\left.\frac{(1-\rho)(1-\gamma-sec\xi)}{q(m-1)}\left[(1-q)-(1-q)^m-q(m-1)(1-q)^m\right]\right\}\leq1-\gamma.
 \end{eqnarray}
\end{theorem}
\begin{proof}
In view of Lemma \ref{lem1}, it is required to show that%
\begin{equation*}
\sum\limits_{n=2}^{\infty }[ (1-\rho)(n-1) \sec \xi +
(1-\gamma) (1+n\rho -\rho)  ]\binom{n+m-2}{m-1}
q^{n-1}(1-q)^m\vert  a_{n}\vert\leq 1-\gamma  .
\end{equation*}

Let
\begin{eqnarray*}
M_3(\xi, \gamma,\rho)&=&\sum\limits_{n=2}^{\infty }[ (1-\rho)(n-1) \sec \xi +
(1-\gamma) (1+n\rho -\rho)  ]\\&\qquad\times&\binom{n+m-2}{m-1}
q^{n-1}(1-q)^m\vert a_{n}\vert. \end{eqnarray*}Since $f\in \mathcal{R} ^{\tau } (\vartheta,\delta),$ then by Lemma \ref{lem4}, we have%
\begin{equation*}
 \left| a_{n} \right| \leq \frac{2\left| \tau \right|(1-\delta) }{1+\vartheta(n-1)}, \quad
   n \in \mathbb{N} \setminus\{1\}
\end{equation*}
and $1 + \vartheta (n - 1)\geq \vartheta n.$ Thus,we have
\begin{eqnarray*}
M_3(\xi, \gamma,\rho)&\leq &\frac{2\left| \tau \right|(1-\delta) }{\vartheta} \left[ \sum\limits_{n=2}^{\infty }\frac{1}{n}[ (1-\rho)(n-1) \sec \xi +
(1-\gamma) (1+n\rho -\rho)  ] \right.\\
&\qquad\times&\left.\binom{n+m-2}{m-1}q^{n-1}(1-q)^m\right]\\
&=&\frac{2\left| \tau \right|(1-\delta) }{\vartheta}(1-q)^m\left[\sum\limits_{n=2}^{\infty }\left[(1-\rho)sec\xi+\rho(1-\gamma)\right] \right.\\
&\qquad+&\left.\left.(1-\rho)(1-\gamma-sec\xi)\frac{1}{n}\right]
\binom{n+m-2}{m-1}q^{n-1}\right].\end{eqnarray*}
Using \eqref{GMS3}, we get
\begin{eqnarray*}
M_3(\xi, \gamma,\rho)&=&\frac{2\left| \tau \right|(1-\delta) }{\vartheta}(1-q)^m \left\{\left[(1-\rho)sec\xi+\rho(1-\gamma)\right]\left[\sum_{n=0}^{\infty}\binom{n+m-1}{m-1}
\,q^{n}-1\right]\right.\\
&\qquad+&\left.\frac{(1-\rho)(1-\gamma)}{q(m-1)}\left[\sum_{n=0}^{\infty}
\binom{n+m-2}{m-2}\,q^{n}-1-(m-1)q\right]\right\}\\
&=&\frac{2\left| \tau \right|(1-\delta) }{\vartheta} \left\{\frac{}{}\left[(1-\rho)sec\xi+\rho(1-\gamma)\right]\left[1-(1-q)^{m}\right]\right.\\
&\qquad+&\left.\frac{(1-\rho)(1-\gamma-sec\xi)}{q(m-1)}\left[(1-q)-(1-q)^m-q(m-1)(1-q)^m\right]\right\}.\\
\end{eqnarray*}
But $M_3(\xi, \gamma,\rho)$ is bounded by $1-\gamma$, if (\ref{d3})
holds. This completes the proof of Theorem \ref{th2}.
\end{proof}
Applying Lemma \ref{lem2} and using the same technique as in the proof of Theorem \ref{th11}, we have the following result.
\begin{theorem}
\label{th2a}If $\ f\in \mathcal{R} ^{\tau } (\vartheta,\delta),\ $ then $\Lambda_q^mf(z)$
is in $\mathcal{K}(\xi, \gamma,\rho)$ if and only if
\begin{eqnarray*}
&&\frac{2\left| \tau \right|(1-\delta) }{\vartheta} \left[[(1-\rho)\sec \xi+(1-\gamma)]\frac{m(m+1)q^{2}}{(1-q)^2}\right.\nonumber\\&+&\left.\left[2(1-\rho)sec\xi+(1-\gamma)(4-\rho)\right] \frac{mq}{1-q}+\left[(1-\gamma)(2 -\rho)\right]\left(1-(1-q)^m \right)\frac{}{}\right]\leq 1-\gamma .
\end{eqnarray*}
\end{theorem}
\section{An integral operator}

\begin{theorem}\label{th3}
If the function $\mathcal{G}^{m}_{q}(z)$ is given by
\begin{equation}\label{io1}
\mathcal{G}^{m}_{q}(z)=\int_{0}^{z}\frac{\Theta_{q}^{m}(t)}{t}dt, \,\,\,z \in \mathbb{D}
\end{equation}%
then $\mathcal{G}^{m}_{q}(z) \in \mathcal{K}(\xi, \gamma,\rho)$ if and only if
\begin{eqnarray*}
[(1-\rho)\sec \xi+\rho(1-\gamma)]\frac{q m}{(1-q)^{m+1}}\leq1-\gamma.
\end{eqnarray*}
\end{theorem}
\begin{proof}
Since%
\begin{equation*}
\mathcal{G}^{m}_{q}(z)=z+\sum_{n=2}^{\infty}\binom{n+m-2}{m-1}%
\,q^{n-1}(1-q)^m\frac{z^n}{n}
\end{equation*}%
then by Lemma \ref{lem2}, we need only to verify that%
\begin{equation*}
\sum\limits_{n=2}^{\infty }n[ (1-\rho)(n-1) \sec \xi +
(1-\gamma) (1+n\rho -\rho)  ]\times\frac{1}{n}\binom{%
n+m-2}{m-1}\,q^{n-1}(1-q)^m\leq 1-\gamma ,
\end{equation*}

or, equivalently%
\begin{equation*}
\sum\limits_{n=2}^{\infty }[ (1-\rho)(n-1) \sec \xi +
(1-\gamma) (1+n\rho -\rho)  ]\binom{n+m-2}{m-1}
q^{n-1}(1-q)^m\leq 1-\gamma .  \label{gg}
\end{equation*}
The remaining part of the proof of Theorem \ref{th3} is similar to that of
Theorem \ref{th1}, and so we omit the details.
\end{proof}

\begin{theorem}
\label{th6}If $m>0,$ then the integral operator $\mathcal{G}^{m}_{q}$%
given by (\ref{io1}) is in $\mathcal{S}(\xi, \gamma,\rho)$ if and only if
~%
\begin{eqnarray*}
&& \frac{}{}\left[(1-\rho)sec\xi+\rho(1-\gamma)\right]\left[1-(1-q)^{m}\right]\\
&+&\frac{(1-\rho)(1-\gamma-sec\xi)}{q(m-1)}\left[(1-q)-(1-q)^m-q(m-1)(1-q)^m\right]\leq1-\gamma.
\end{eqnarray*}
\end{theorem}
\begin{proof}
Since%
\begin{equation*}
\mathcal{G}^{m}_{q}(z)=z+\sum_{n=2}^{\infty}\binom{n+m-2}{m-1}%
\,q^{n-1}(1-q)^m\frac{z^n}{n}
\end{equation*}%
then by Lemma \ref{lem1}, we need only to verify that%
\begin{equation*}
\sum\limits_{n=2}^{\infty }\frac{ 1}{n}[ (1-\rho)(n-1) \sec \xi +
(1-\gamma) (1+n\rho -\rho)  ]\binom{n+m-2}{m-1}
q^{n-1}(1-q)^m\leq 1-\gamma  .
\end{equation*}
Thus, we have
\begin{eqnarray*}
M_4(\xi, \gamma,\rho)&=&\sum\limits_{n=2}^{\infty }\frac{1}{n}[ (1-\rho)(n-1) \sec \xi +
(1-\gamma) (1+n\rho -\rho)  ]\\
&\qquad\times&.\binom{n+m-2}{m-1}q^{n-1}(1-q)^m\\
&=&(1-q)^m\left[\sum\limits_{n=2}^{\infty }\left[(1-\rho)sec\xi+\rho(1-\gamma)\right] \right.\\
&\qquad+&\left.\left.(1-\rho)(1-\gamma-sec\xi)\frac{1}{n}\right]\binom{n+m-2}{m-1}q^{n-1}\right].\end{eqnarray*}
Using \eqref{GMS3}, we get
\begin{eqnarray*}
M_4(\xi, \gamma,\rho)&=&(1-q)^m \left\{\left[(1-\rho)sec\xi+\rho(1-\gamma)\right]\left[\sum_{n=0}^{\infty}\binom{n+m-1}{m-1}
\,q^{n}-1\right]\right.\\
&+&\left.\frac{(1-\rho)(1-\gamma-sec\xi)}{q(m-1)}\left[\sum_{n=0}^{\infty}
\binom{n+m-2}{m-2}\,q^{n}-1-(m-1)q\right]\right\}\\
&=& \left\{\frac{}{}\left[(1-\rho)sec\xi+\rho(1-\gamma)\right]\left[1-(1-q)^{m}\right]\right.\\
&+&\left.\frac{(1-\rho)(1-\gamma-sec\xi)}{q(m-1)}\left[(1-q)-(1-q)^m-q(m-1)(1-q)^m\right]\right\}.\\
\end{eqnarray*}
But $M_4(\xi, \gamma,\rho)$ is bounded by $1-\gamma$, if (\ref{d3})
holds. This completes the proof of Theorem \ref{th2}.
\end{proof}
The proof of Theorem \ref{th6} is similar to that of Theorem \ref%
{th3}, so we omitted the proof of Theorem \ref{th6}.
\section{Corollaries and consequences}
By taking $\rho=0$ in Theorems \ref{th1}-\ref{th6}, we
obtain the necessary and sufficient condition for Pascal distribution series be in the classes $\mathcal{S}(\xi, \gamma)$ and $\mathcal{K}(\xi, \gamma)$ from the following corollaries.

\begin{corollary}\label{pc-cor1}
If $m>0,$ then $\Theta^{m}_{q}$ is in $\mathcal{S}(\xi, \gamma)$ if and only if
\begin{eqnarray*}
\frac{q m\sec \xi}{(1-q)^{m+1}}\leq1-\gamma.
 \end{eqnarray*}
\end{corollary}
\begin{corollary}\label{pc-cor2}
If $m>0,$ then $\Theta^{m}_{q}$is in $\mathcal{K}(\xi, \gamma)$ if and only if

\begin{eqnarray*}\label{lp}
&&[\sec \xi+(1-\gamma)]\frac{m(m+1)q^{2}}{(1-q)^2}+\left[2sec\xi+4(1-\gamma)\right] \frac{mq}{1-q}\nonumber\\
&+&\left[2(1-\gamma)\right]\left(1-(1-q)^m \right)\leq 1-\gamma.
 \end{eqnarray*}
\end{corollary}
\begin{corollary}\label{pc-cor3}
If $\ f\in \mathcal{R} ^{\tau } (\vartheta,\delta)\ $then $\Lambda_q^m$
is in $\mathcal{S}(\xi, \gamma)$ if and only if
\begin{eqnarray*}
&&\frac{2\left| \tau \right|(1-\delta) }{\vartheta} \left\{\frac{}{}sec\xi\left[1-(1-q)^{m}\right]\right.\\
&+&\left.\frac{(1-\gamma-sec\xi)}{q(m-1)}\left[(1-q)-(1-q)^m-q(m-1)(1-q)^m\right]\right\}\leq1-\gamma.
\end{eqnarray*}
\end{corollary}
\begin{corollary}\label{pc-cor4}
If $\ f\in \mathcal{R} ^{\tau } (\vartheta,\delta),\ $then $\Lambda_q^m$
is in $\mathcal{K}(\xi, \gamma)$ if and only if
\begin{eqnarray*}
&&\frac{2\left| \tau \right|(1-\delta) }{\vartheta} \left[[\sec \xi+(1-\gamma)]\frac{m(m+1)q^{2}}{(1-q)^2}+\left[2sec\xi+4(1-\gamma)\right] \frac{mq}{1-q}\right.\nonumber\\&+&\left.\left[2(1-\gamma)\right]\left(1-(1-q)^m \right)\frac{}{}\right]\leq 1-\gamma .
\end{eqnarray*}
\end{corollary}
\begin{corollary}\label{pc-cor5}
If $m>0,$ then the integral operator $\mathcal{G}^{m}_{q}(z)$%
given by (\ref{io1}) is in $\mathcal{K}(\xi, \gamma)$ if and only if
\begin{eqnarray*}
\frac{q m\sec \xi}{(1-q)^{m+1}}\leq1-\gamma.
\end{eqnarray*}
\end{corollary}
\begin{corollary}\label{pc-cor6}
If $m>0,$ then the integral operator $\mathcal{G}^{m}_{q}(z)$%
given by (\ref{io1}) is in $\mathcal{S}(\xi, \gamma)$ if and only if
~%
\begin{eqnarray*}
&&sec\xi\left[1-(1-q)^{m}\right]\\
&+&\frac{(1-\gamma-sec\xi)}{q(m-1)}\left[(1-q)-(1-q)^m-q(m-1)(1-q)^m\right]\leq1-\gamma.
\end{eqnarray*}
\end{corollary}

\end{document}